\documentclass[a4paper, 10pt]{amsart}

\title{Embedding universal covers of graph manifolds in products of trees}
\author{David Hume and Alessandro Sisto}

\usepackage{amssymb} \usepackage{amsfonts} \usepackage{amsmath}
\usepackage{amsthm, hyperref}

\newtheorem{lemma}{Lemma}
\newtheorem{thm}[lemma]{Theorem}

\newtheorem{cor}[lemma]{Corollary}

\newtheorem{qu}[lemma]{Question}

\newcommand{\tilM}{\ensuremath{{\widetilde{M}}}}

\newcommand{\g}{\gamma}
\newcommand{\R}{\mathbb{R}}
\newcommand{\N}{\mathbb{N}}
\newcommand{\matH}{\mathbb{H}}

\address{Mathematical Institute, 24-29 St Giles, Oxford OX1 3LB, United Kingdom}

\email{hume@maths.ox.ac.uk, sisto@maths.ox.ac.uk}

\begin{document}

\begin{abstract}
We prove that the universal cover of any graph manifold quasi-isometrically embeds into a product of three trees. In particular we show that the Assouad-Nagata dimension of the universal cover of any closed graph manifold is $3$, proving a conjecture of Smirnov.
\end{abstract}
\maketitle

A graph manifold is a compact connected $3$-manifold (possibly with boundary) which admits a decomposition into Seifert fibred pieces, when cut along a collection of embedded tori and/or Klein bottles. In particular a graph manifold is a $3$-manifold whose geometric decomposition admits no hyperbolic part. For this reason the class of graph manifold groups is rigid within the class of $3$-manifold groups \cite{KL-qiclasses}, moreover, such groups are classified up to quasi-isometry \cite{BN09}. 
\par 
More details on graph manifolds and proofs of the above results can be found in \cite{BDM09}, \cite{Ge94} and \cite{KL98}. 
\par
We show the following:

\begin{thm} \label{gm} \hspace{3mm}
 The universal cover of any graph manifold quasi-isometrically embeds in the product of three metric trees.
\end{thm}
One may wish to compare this theorem with the result by Buyalo and Schroeder \cite{BS05} that $\matH^3$ can be quasi-isometrically embedded in the product of three infinite valence simplicial trees. (This was refined to three infinite binary trees by \cite{BDS07}.)
\par
As an application, we determine the Assouad-Nagata dimension ($dim_{AN}$) - as defined by Assouad, \cite{As82} - of the universal cover of closed graph manifolds. We denote the asymptotic Assouad-Nagata dimension by $asdim_{AN}$. Recall that the Assouad-Nagata dimension bounds from above the asymptotic dimension, first introduced by Gromov in \cite{Gr93}. However, asymptotic dimension and asymptotic Assouad-Nagata dimension can differ radically, see for instance the examples in \cite{BDU06}. The asymptotic Assouad-Nagata dimension of a group also bounds from above the dimension of its asymptotic cones \cite{DyHi} and if a group has finite Assouad-Nagata dimension then it has compression exponent $1$ \cite{Gal}.
\par
The asymptotic dimension of universal covers of closed graph-manifolds is known to be $3$, as mentioned in \cite{Sm10}, in view of results in \cite{BD08}. Also, Smirnov \cite{Sm10} showed that their Assouad-Nagata dimension is finite (at most $7$) and conjectured that it actually equals $3$. Theorem \ref{gm} implies his conjecture:

\begin{cor}\label{AN-dim} \hspace{3mm}
 If $\tilM$ is the universal cover of a closed graph-manifold then
$$dim_{AN} \tilM=asdim_{AN} \tilM=3.$$
\end{cor}
\begin{proof}
 Asymptotic dimension never exceeds either of the aforementioned dimensions, so this provides the lower bound of $3$ in both cases, also, by definition $dim_{AN} \tilM\leq asdim_{AN} \tilM$. Results in \cite{LS05} prove $asdim_{AN} X\leq n$ when $X$ is an $n-$fold product of trees and $asdim_{AN} A\leq asdim_{AN} B$ whenever $A$ admits a quasi-isometric embedding into $B$, so we get the upper bound using Theorem \ref{gm}.
\end{proof}

A graph manifold is said to be \emph{non-geometric} if its decomposition into Seifert fibred pieces is non-trivial. Notice that if the decomposition is trivial then the universal cover is quasi-isometric to the product of a tree with $\R$.

\begin{qu}
 Can fundamental groups of non-closed, non-geometric graph manifold quasi-isometrically embed into a product of $2$ trees\textup{?}
\end{qu}

\subsection*{Proof of Theorem \ref{gm}} \hspace{3mm}
 We only have to consider non-geometric \emph{flip} graph manifolds. In fact - at the level of universal covers - any graph manifold is quasi-isometric to a flip graph manifold \cite{KL98}. We do not need the definition of such manifolds, as we will recall the essential properties required. Let $M$ be a flip graph manifold and let $T$ be its Bass-Serre tree. The universal cover $\tilM$ of $M$ is constructed by suitably gluing certain metric spaces $X_v=F_v\times\R$, for $v$ a vertex in $T$. Each $F_v$ is the universal cover of a compact surface with non-empty boundary and so it admits a metric retraction $r_v:F_v\to T_v$, where $T_v\subseteq F_v$ is a tree, with the further properties that $r_v$ is injective when restricted to any boundary component of $F_v$ and there exists $\mu$ (not depending on $v$) such that for each $x\in F_v$ we have $d_{F_v}(x,r_v(x))\leq \mu$. Finally, the gluings are performed as follows. Let $v,v'$ be adjacent vertices. Then there exist parametrisations $\gamma_v:\R\to F_v, \gamma_{v'}:\R\to F_{v'}$ of boundary components of $F_v, F_{v'}$ so that $(\gamma_{v}(t),u)\in F_{v}\times\R$ is identified with $(\gamma_{v'}(u),t)\in F_{v'}\times\R$ for each $t,u\in\R$. This is explained, for example, in \cite{BN09}.
\par
 \noindent {\bf Step 1. The trees.} \par
The first tree will just be the Bass-Serre tree $T_0=T$. Let us define the other two trees $T_1, T_2$ as follows. \par
We can subdivide the vertices of $T$ into disjoint families $V_1, V_2$ such that if $v,v'\in V_i$ then $d_T(v,v')$ is even. Set $T'_i=\bigsqcup_{v\in V_i} T_v$. We wish now to define an equivalence relation $\sim$ on $T'_i$, and we will set $T_i=T'_i/_\sim$. Suppose that $v,v'\in V_i$, $v\neq v'$ and there exists $w$ such that $d_T(v,w)=d_T(v',w)=1$. We will set $x\sim_d x'$, for $x\in T_v$, $x'\in T_{v'}$, if there exist $y,y'$ with $r_v(y)=x,r_{v'}(y')=x'$ such that the points in $F_{w}\times\R$ identified with $(y,0)\in F_v\times\R$, $(y',0)\in F_{v'}\times \R$ have the same $\R-$coordinate. To ensure an equivalence relation, we set $\sim$ to be the transitive closure of $\sim_d$.
\par
It is very easy to check that $T_i=T'_i/_\sim$ is a metric tree with only countably many branching points. In fact, it can be described as the increasing union of metric spaces $\{X_k\}_{k\in\N}$ such that $X_0$ is a tree and $X_{k+1}$ is obtained from $X_k$ by identifying a line in $X_k$ with a line in some tree.
\par
 \noindent {\bf Step 2. The components of the embedding.}\par
Define $f_0:\tilM\to T_0$ to be any map such that for all $x\in \tilM$, $x\in F_{f_0(x)}\times\R$ and define $f_i:\tilM\to T_i$ as follows. For each $v$, we let $\pi_v:F_v\times \R\to F_v$ be the projection on the first factor, and as usual denote the equivalence classes of $\sim$ with square brackets. \par If $x\in F_{v}\times\R$ for some $v\in V_i$, then set $f_i(x)=[r_v(\pi_{v}(x))]$. Otherwise we have $x\in F_w\times\R$ for $w\notin V_i$. Let $v\in V_i$ be any vertex such that $d_T(v,w)=1$. Set $f_i(x)=[p]$ where $p\in T_{v}$ is such that $(p,0)$ has, as a point in $F_{w}\times\R$, the same $\R-$coordinate as $x$. This does not depend on the choice of $v$, by the equivalence relation.
\par
 \noindent {\bf Step 3. The product map is a quasi-isometric embedding.}\par
 Define $f:\tilM\to \prod T_i$ to be $\prod f_i$. We wish to show that $f$ is a quasi-isometric embedding. The easier inequality is $d(f(x),f(y))\leq K d(x,y)+C$: the maps $\pi_{v}$ and $r_v$ are non-expanding, so $f_1$ and $f_2$ are readily checked to be $1-$Lipschitz, while $f_0$ satisfies $d_{T_0}(f_0(x),f_0(y))\leq d_\tilM(x,y)/\rho+1$ where
$$0<\rho=\inf \{d_\tilM(x,x'):x\in X_{v},x'\in X_{v'}, d_{T_0}(v,v')=2\}.$$
For the other inequality we start with a geodesic $\delta$ in $\prod T_i$ connecting $f(x)$ to $f(y)$ and construct a path $\g$ in $\tilM$ connecting $x$ to $y$ such that $l(\g)\leq K l(\delta)+C$. Let $\delta_1, \delta_2$ be the projections of $\delta$ on the factors. One may wish to compare the paths we obtain in this way with the ``special paths'' described in \cite{Si11}. 
\par
Suppose that $x\in X_{v_0}$, $y\in X_{v_n}$ and let $v_0,\dots,v_n$ be the vertices of $T$ in the geodesic connecting $v_0$ to $v_n$. For $j=0,\dots, n$ let $i(j)\in\{1,2\}$ be such that $v_j\in V_{i(j)}$ and choose $\alpha_j\subseteq \delta_{i(j)}$ so that $\alpha_j\subseteq \left[r_{v_j}\left(F_{v_j}\right)\right]$. We will also require that the final point of $\alpha_j$ is the starting point of $\alpha_{j+2}$, that the starting point of $\alpha_0$ is $f_{i(0)}(x)$ and that the final point of $\alpha_n$ is $f_{i(n)}(y)$. This can be easily arranged using the fact that each $[r_v(F_v)]$ is convex in the corresponding $T_i$.
\par
For $j=0,\dots, n-1$, let $t_j$ be the $\R-$coordinate as a point in $F_{v_j}\times\R$ of $(p_j,0)\in F_{v_{j+1}}\times \R$, where $p_j$  is the starting point of $\alpha_{j+1}$. Also, let $t_n$ be the $\R-$coordinate of $y\in F_{v_n}\times\R$.
\par
For $j=0,\dots,n$ let $\gamma_j$ be the path $\alpha_j\times t_j$ in $X_{v_j}$. Notice that the distance between the final point of $\gamma_j$ and the starting point of $\gamma_{j+1}$ is at most $2\mu$. So, we can concatenate in a suitable order the $\gamma_j$'s and $n$ geodesics of length at most $2\mu$ to obtain a path $\gamma$ from $x$ to $y$. Clearly $l(\gamma_j)=l(\alpha_j)$ so 
$$l(\g)\leq \sum l(\g_j)+2n\mu=l(\delta_1)+l(\delta_2)+2n\mu=$$
$$d(f_1(x),f_1(y))+d(f_2(x),f_2(y))+2n\mu.$$
As $d(f_0(x),f_0(y))\geq n-2$ we have 
$$l(\g)\leq d(f_1(x),f_1(y))+d(f_2(x),f_2(y))+2\mu d(f_0(x),f_0(y))+4\mu,$$
and we are done. \qed

\bibliographystyle{alpha}
\bibliography{DB}

\begin{thebibliography}{BDM09}

\bibitem[Ass82]{As82}
P.~Assouad.
\newblock Sur la distance de {N}agata.
\newblock {\em C. R. Acad. Sci. Paris S\'er. I Math.}, 294(1):31--34, 1982.

\bibitem[BD08]{BD08}
G.~Bell and A.~Dranishnikov.
\newblock Asymptotic dimension.
\newblock {\em Topology Appl.}, 155(12):1265--1296, 2008.

\bibitem[BDL06]{BDU06}
N.~Brodskiy, J.~Dydak, and U.~Lang.
\newblock Assouad-{N}agata dimension of wreath products of groups.
\newblock arXiv.org:math.MG/0611331, 2006.

\bibitem[BDM09]{BDM09}
J.~Behrstock, C.~Dru{\c{t}}u, and L.~Mosher.
\newblock Thick metric spaces, relative hyperbolicity, and quasi-isometric
  rigidity.
\newblock {\em Math. Ann.}, 344(3):543--595, 2009.

\bibitem[BDS07]{BDS07}
S.~Buyalo, A.~Dranishnikov, and V.~Schroeder.
\newblock Embedding of hyperbolic groups into products of binary trees.
\newblock {\em Invent. Math.}, 169(1):153--192, 2007.

\bibitem[BN08]{BN09}
J.~Behrstock and W.~Neumann.
\newblock Quasi-isometric classification of graph manifold groups.
\newblock {\em Duke Math. J.}, 141(2):217--240, 2008.

\bibitem[BS05]{BS05}
S.~Buyalo and V.~Schroeder.
\newblock Embedding of hyperbolic spaces in the product of trees.
\newblock {\em Geom. Dedicata}, 113:75--93, 2005.

\bibitem[DH08]{DyHi}
J.~Dydak and J.~Higes.
\newblock Asymptotic cones and {A}ssouad-{N}agata dimension.
\newblock {\em Proc. Amer. Math. Soc.}, 136(6):2225--2233, 2008.

\bibitem[Gal08]{Gal}
{\'S}wiatos{\l}aw~R. Gal.
\newblock Asymptotic dimension and uniform embeddings.
\newblock {\em Groups Geom. Dyn.}, 2(1):63--84, 2008.

\bibitem[Ger94]{Ge94}
S.~Gersten.
\newblock Divergence in {$3$}-manifold groups.
\newblock {\em Geom. Funct. Anal.}, 4(6):633--647, 1994.

\bibitem[Gro93]{Gr93}
M.~Gromov.
\newblock Asymptotic invariants of infinite groups.
\newblock In {\em Geometric group theory, {V}ol.\ 2 ({S}ussex, 1991)}, volume
  182 of {\em London Math. Soc. Lecture Note Ser.}, pages 1--295. Cambridge
  Univ. Press, Cambridge, 1993.

\bibitem[KL95]{KL-qiclasses}
M.~Kapovich and B.~Leeb.
\newblock On asymptotic cones and quasi-isometry classes of fundamental groups
  of {$3$}-manifolds.
\newblock {\em Geom. Funct. Anal.}, 5(3):582--603, 1995.

\bibitem[KL98]{KL98}
M.~Kapovich and B.~Leeb.
\newblock {$3$}-manifold groups and nonpositive curvature.
\newblock {\em Geom. Funct. Anal.}, 8(5):841--852, 1998.

\bibitem[LS05]{LS05}
U.~Lang and T.~Schlichenmaier.
\newblock Nagata dimension, quasisymmetric embeddings, and {L}ipschitz
  extensions.
\newblock {\em Int. Math. Res. Not.}, (58):3625--3655, 2005.

\bibitem[Sis11]{Si11}
A.~Sisto.
\newblock 3-manifold groups have unique asymptotic cones.
\newblock {\em Preprint
  \href{http://www.arxiv.org/abs/1109.4674}{\tt{arXiv:1109.4674}}}, 2011.

\bibitem[Smi10]{Sm10}
A.~Smirnov.
\newblock The linearly controlled asymptotic dimension of the fundamental group
  of a graph manifold.
\newblock {\em Algebra i Analiz}, 22(2):185--203, 2010.

\end{thebibliography}

\end{document}